\documentclass[12pt]{amsart}

\usepackage{ulem}

\usepackage{amsmath,amssymb,amsthm,enumitem}

\usepackage{multirow}

\usepackage{tikz-cd}

\usepackage{amsfonts}

\usepackage{tikz}
\usetikzlibrary{matrix,arrows,decorations.pathmorphing}

\usepackage{enumitem}

\makeatletter
\@namedef{subjclassname@2010}{%
  \textup{2010} Mathematics Subject Classification}
\makeatother

\DeclareMathOperator{\Ker}{Ker}
\DeclareMathOperator{\Ima}{Im}
\DeclareMathOperator{\lcm}{lcm}
\DeclareMathOperator{\Proj}{Proj}
\DeclareMathOperator{\Spec}{Spec}

\newtheorem{theorem}{Theorem}[section]

\newtheorem{proposition}[theorem]{Proposition}
\newtheorem{lemma}[theorem]{Lemma}

\newtheorem{conjecture}[theorem]{Conjecture}
\theoremstyle{definition}
\newtheorem{remark}[theorem]{Remark}

\def\cqfd{
{\hfill
\kern 6pt\penalty 500
\raise -1pt\hbox{\vrule\vbox to 5pt{\hrule width 4pt
\vfill\hrule}\vrule}}
\break}

\font\tengoth=eufm10
\font\sevengoth=eufm7
\font\fivegoth=eufm5
\newfam\gothfam
\textfont\gothfam=\tengoth\scriptfont\gothfam=\sevengoth\scriptscriptfont\gothfam=\fivegoth

\title[Hypersurfaces in weighted projective spaces]{Maximum number of rational points on hypersurfaces in weighted projective spaces over finite fields}

\author[Yves Aubry]{Yves Aubry}
\address[Aubry]{Institut de Math\'ematiques de Toulon - IMATH, Universit\'e de Toulon, France}
\address[Aubry]{Institut de Math\'ematiques de Marseille - I2M, Aix Marseille Univ, UMR 7373 CNRS, France}
\email{yves.aubry@univ-tln.fr}

\author[Marc Perret]{Marc Perret}
\address[Perret]{Institut de Math\'ematiques de Toulouse - UMR 5219, CNRS, UT2J, F-31058 Toulouse, France}
\email{perret@math.univ-toulouse.fr}

\date{\today}

\begin{document} 

\baselineskip=17pt


\begin{abstract}
An upper bound for the maximum number of rational points on a hypersurface in a projective space over a finite field has been conjectured by Tsfasman and proved by Serre in 1989.
The analogue question for hypersurfaces on weighted projective spaces has been considered by  Castryck,  Ghorpade,  Lachaud,  O'Sullivan,  Ram and the first author in 2017. A conjecture has been proposed there 
under the assumption that the first weight is equal to one
and proved in the particular case of the dimension 2.
We  prove here the conjecture in any dimension provided the second weight is also equal to one.
\end{abstract}



\maketitle

\begin{center}
\sl Dedicated to  our friend Sudhir Ghorpade for his 60$^{th}$ birthday\footnote{This work is partially supported by the French Agence Nationale de la Recherche through
the BARRACUDA project under Contract ANR-21-CE39-0009.}.
\end{center}

\section{Introduction}

Let 
${\mathbb F}_q$ be the finite field with $q$ elements and ${\mathbb P}^n({\mathbb F}_q)$ be the set of rational points over ${\mathbb F}_q$ of the projective space of dimension $n\geq 1$.
Let us set $p_n:=q^n+\cdots +q+1$ for $n\geq 0$ and $p_n:=0$ for $n<0$.
We have clearly $\sharp {\mathbb P}^n({\mathbb F}_q)=p_n$.

Answering  a conjecture that Tsfasman made at the ``Journ\'ees Arith\-m\'etiques de Luminy'' in 1989, Serre proved in \cite{Serre} (and independently  S\o rensen proved later  in \cite{Sorensen}) that if $F$ is a nonzero homogeneous polynomial in ${\mathbb F}_q[X_0,\ldots,X_n]$
of degree $d\geq 1$, then the number of rational points over ${\mathbb F}_q$ of the hypersurface $V(F)$ in
${\mathbb P}^n$ defined by $F$ satisfies the so-called Serre bound:
$$\sharp V(F)({\mathbb F}_q)\leq dq^{n-1}+p_{n-2}.$$
If   $d\geq q+1$ then $dq^{n-1}+p_{n-2}\geq p_n=\sharp {\mathbb P}^n({\mathbb F}_q)$  and  
the hypersurface defined by 
the degree $d$ homogeneous polynomial $X_0^{d-q-1}(X_0^qX_1-X_0X_1^q)$  has $p_n$ rational points.
Thus the Serre bound holds
trivially and is reached for hypersurfaces of degree greater than or equal to $q+1$.

Furthermore,  the Serre bound is reached for hypersurfaces of degree less than or equal to $q$. Indeed, 
if $d\leq q$ then the number of rational points on the hypersurface given by the polynomial 
  $$F=\prod_{i=1}^d(\alpha_iX_0-\beta_iX_1),$$
where $(\alpha_1:\beta_1), \ldots, (\alpha_d,\beta_d)$ are distincts elements of ${\mathbb P}^1({\mathbb F}_q)$, attains the Serre bound.
Note that Serre proved that the  bound is
reached for $d\leq q$ if and only if $F$ is of the above form, that is $V(F)$ is the union of $d$ hyperplanes containing a linear variety of codimension 2.




In 1997, Tsfasman and Boguslavsky in \cite{Bogu} have considered the analogue question for a system of $r$ polynomial equations.
They propose a conjecture for the maximum number of points in ${\mathbb P}^n({\mathbb F}_q)$ of  the projective set given by the common zeros of $r$ linearly independent homogeneous polynomials of degree $d$ in ${\mathbb F}_q[X_0,\ldots,X_n]$.
The Tsfasman-Boguslavsky conjecture for $r=1$ is nothing else but the Serre bound. Boguslavsky succeded to prove  in \cite{Bogu} the case  $r=2$.
In 2015, Datta and Ghorpade proved in \cite{D-G-2015} that the Tsfasman-Boguslavsky  conjecture is true if $d=2$ and $r\leq n+1$ but is false in general if $d=2$ and $r\geq n+2$. Moreover, in 2017 they proved in \cite{D-G-2017} that the Tsfasman-Boguslavsky  conjecture is true for any positive integer $d$, provided $r\leq n+1$.
The case for $r$ beyond $n+1$ is specifically considered one year later by Beelen, Datta and Ghorpade in \cite{B-D-G-2018} and they conjectured in 2022 in \cite{B-D-G-2022} a general formula when $d<q$ that they were able to prove in some cases




\medskip

We are interested here in a generalization in another direction, namely the question of Tsfasman and Serre in the context of weighted projective spaces ${\mathbb P}(a_0,\ldots,a_n)$, i.e.
 the study, for any homogeneous polynomial $F$ in ${\mathbb F}_q[X_0,\ldots,X_n]$ of degree $d$ (with respect to the weights $a_0,\ldots,a_n$), of the maximum number of  rational points on the hypersurface $V(F)$  in ${\mathbb P}(a_0,\ldots,a_n)$. 
In \cite{UCLA}, the following quantity has been introduced:
$$e_q(d ; a_0,a_1,a_2,\ldots,a_n):=\max_{F}\sharp V(F)({\mathbb F}_q)$$
where the maximum ranges over the set of homogeneous polynomials $F$ in ${\mathbb F}_q[X_0,\ldots,X_n]$ of weighted degree $d$.

It has been conjectured in 2017  in \cite{UCLA} that:

\begin{conjecture}\label{conjecture}
If $a_0=1$ and $\lcm(a_1,a_2,\ldots,a_n)\vert d$, and if we order the weights such that $a_1\leq a_2\leq \ldots \leq a_n$ then
$$e_q(d;1,a_1,a_2,\ldots,a_n)=\min\{p_n,\frac{d}{a_1}q^{n-1}+p_{n-2}\}.$$
\end{conjecture}

In the case of the projective line ${\mathbb P}(a_0,a_1)$, it has been shown  in \cite{UCLA} that $e_q(d; a_0,a_1)=\min\{p_1, d/a\}$ where $a=\lcm(a_0,a_1)$, so the conjecture holds in this case.
Moreover, the conjecture has been proved  in \cite{UCLA} 
for projective planes ${\mathbb P}(1,a_1,a_2)$ with $a_1$ and $a_2$ coprime and $a_1<a_2$: if $F\in{\mathbb F}_q[X_0,X_1,X_2]$ is a nonzero weighted homogenous polynomial of degree $d\leq a_1(q+1)$ which is a multiple of $a_1a_2$ then $\sharp V(F)({\mathbb F}_q) \leq \frac{d}{a_1}q+1$.
The proof follows the one given by Serre  with a new notion of lines represented by either a homogenized linear bivariate equation, or  the line at infinity.

Our purpose here is to prove Conjecture \ref{conjecture} in any dimension $n$
provided $a_1=1$.

We recall in Section \ref{section-lower-bound}  the basic facts about weighted projective spaces and a lower bound for $e_q(d;a_0,\ldots,a_n)$.
Then we study in Section \ref{section-pullback} some morphisms between weighted projective spaces and we establish a relation between  
the numbers of zeros of a polynomial and its pullback.
Section \ref{section-upper-bound} is devoted to the proof of an upper bound for the number of rational points on an hypersurface in a weighted projective space.
Finally we state and prove the main result in Section \ref{section-main-result}.




\section{A lower bound for the number of rational points}\label{section-lower-bound}



\subsection{Weighted projective spaces}\label{section-wps}

Let $a_0,\ldots, a_n$ be positive integers  and $S$ be the polynomial ring ${\mathbb F}_q[X_0,\ldots,X_n]$ graded by $\deg(X_i)=a_i$. 
The weighted projective space
${\mathbb P}(a_0,\ldots,a_n)$ over ${\mathbb F}_q$ 
is the scheme
$${\mathbb P}(a_0,\ldots,a_n)=\Proj S,$$
and can be seen as the geometric quotient 
$${\mathbb A}_{{\mathbb F}_q}^{n+1}\setminus \{0\}/{\mathbb G}_{m, {\mathbb F}_q}$$
of the punctured affine space ${\mathbb A}_{{\mathbb F}_q}^{n+1}\setminus \{0\}$ over ${\mathbb F}_q$ under the action of the multiplicative group ${\mathbb G}_{m, {\mathbb F}_q}$ over ${\mathbb F}_q$ given for  any nonzero $\lambda$ in an algebraic closure $\overline{\mathbb{F}}_q$ of ${\mathbb F}_q$ by
$$\lambda . (x_0,\ldots,x_n)=(\lambda^{a_0}x_0,\ldots,\lambda^{a_n}x_n).$$
If the $a_i$'s are all equal to 1, then we recover the usual (or straight) projective space:
${\mathbb P}(1,\ldots,1)={\mathbb P}^n.$

The corresponding equivalent class is denoted by $[x_0:\cdots:x_n]$ without any reference to the corresponding weights $a_0, \ldots, a_n$ and is called a weighted projective point.
We say that the point is ${\mathbb F}_q$-rational if $[x_0:\cdots:x_n]=[x_0^q:\cdots:x_n^q]$.
Every ${\mathbb F}_q$-rational point of a weighted projective space over ${\mathbb F}_q$ has at least one representative  in ${\mathbb F}_q^{n+1}\setminus \{(0,\ldots,0)\}$. 
This result  has been quoted in \cite{Marc} but without a complete proof. 
Due to a lack of proof writing, we provide the following one over any field $k$ which has been communicated to the authors by Laurent Moret-Bailly.
\begin{proposition}\label{LMB-projPond} Let $k$ be a field and $\underline{a}=(a_0, a_1, \cdots, a_n)$ be a sequence of $n+1$ nonzero integers. Then each $k$-rational point $x\in {\mathbb P}(a_0,\ldots,a_n)$ has a representative $x=[x_0:x_1: \cdots: x_n]$ with $x_i\in k$ for any $0\leq i\leq n$.
\end{proposition}

\begin{proof}

[Communicated by Laurent Moret-Bailly] Given a geometric point $x=[x_i ; 0\leq i\leq n]\in {\mathbb P}(a_0,\ldots,a_n)$, we denote by $\vert x\vert :=\{i\in \{0, \cdots, n\}, x_i\neq 0\}$ the support of $x$. Then the whole projective space is partitioned into
$${\mathbb P}(a_0,\ldots,a_n)=\bigcup_{\emptyset \neq I \subset  \{0, \cdots, n\}} W_{\underline{a}}^I,$$
where $W_{\underline{a}}^I := \{x \in {\mathbb P}(a_0,\ldots,a_n), \vert x\vert =I\}$, so that we have to prove that for any nonempty subset $I$ of $\{0, \cdots, n\}$, any $k$-rational point in $W_{\underline{a}}^I$ admits a $k$-rational representative. For this purpose, consider the puncturing
 regular map defined over $k$
$$ 
\begin{matrix} W_{\underline{a}}^I &\longrightarrow &{\mathbb P}(a_i, i\in I)\\
[x_i ; 0\leq i\leq n] & \mapsto & [x_i ; i\in I]
\end{matrix}
$$
into a weighted projective space of dimension $\sharp I-1$. This map is injective
and, in case $\sharp I \geq 2$, is an isomorphism onto the dense torus 
$$T_{(a_i, i\in I)} := \{[x_i; i\in I] \in {\mathbb P}(a_i, i\in I) ; \forall i \in I, x_i \neq 0\}$$
 of ${\mathbb P}(a_i, i\in I)$. 
 Now if $d_I$ denotes  the gcd of the $a_i, i\in I$ and
 $b_i:= \frac{a_i}{d_I}$ for all $i\in I$,
  then we have first that the $b_i, i\in I$ are coprime, second that ${\mathbb P}(a_i, i\in I)$ is $k$-isomorphic to ${\mathbb P}(b_i, i\in I)$ (see the lemma in section 1.1. of \cite{Dolgachev}). Hence, $W_{\underline{a}}^I$ is $k$-isomorphic to the dense torus
$T_{(b_i, i\in I)}$
  of ${\mathbb P}(b_i, i\in I)$ and we are reduced to prove the proposition only for $x$ in the dense torus of a weighted projective space ${\mathbb P}(b_i, i\in I)$ whose weights $(b_i, i\in I)$ are coprime.

To do this, let $(u_i ; i\in I) \in {\mathbb Z}^I$ such that $\sum_{i\in I} u_ib_i=1$, and consider the subset 
$$V_I = \{(x_i ; i\in I)\in {\mathbb A}_{{\mathbb F}_q}^{I}\setminus \{0_I\} ; \quad \prod_{i\in I}  x_i^{u_i}=1\}$$
of the affine space of
dimension  $\sharp I$. It is then easily checked that for any $x=[x_i, i\in I]$ in the dense torus of ${\mathbb P}(b_i, i\in I)$, its only representative $(\lambda^{b_i}x_i, i\in I)$ lying on 
$V_{I}$, for $\lambda \in \overline{k}$, is the one for $\lambda = \prod_{i\in I}x_i^{-u_i}$. This proves that $W_{\underline{a}}^I$, the dense torus $T_{(a_i, i\in I)}$ and the affine subvariety $V^I$ are $k$-isomorphic, and we are done in case $\sharp I\geq 2$. 

In case $\sharp I=1$, the weighted projective space with only one weight $a\in {\mathbb N}^*$ is ${\mathbb P}(a) = {\mathbb A}_{k}^*/{\mathbb G}_{m,k}$ for the action $\lambda . x=\lambda^a x$, so that for any $x\in \overline{k}^*$, we have $[x]=[1]$ (take $\lambda$ be any $a$-th rooth of $x$ in $\overline{k}^*$), so that any point in ${\mathbb P}(a)$ is $k$-rational 
which concludes the proof.
\end{proof}

Furthermore, Laurent Moret-Bailly has communicated to us the following  more general scheme theoretic statement.

\begin{proposition}\label{Prop-LMB_Scheme}(Moret-Bailly)
Let $S=\oplus_{n\geq 0}S_n$ be a positively graded ring. Let $X=\Proj(S)$, $C=\Spec(S)\setminus \Spec(S_0)$ be the punctured cone, and $\rho : C \longrightarrow X$ be the natural projection.
Then, for any $x:\Spec(k) \longrightarrow X$, the reduced fiber $(C\times_{\rho, X,x} \Spec(k))_{red}$ is isomorphic to
$\Spec(k[t,t^{-1}])$.

In particular, for any field $k$, the map $C(k)\longrightarrow X(k)$ induced by $\rho$ is surjective.
\end{proposition}

\begin{proof}
Let $x:\Spec(k) \longrightarrow X$ be a rational point of $X$ over $k$. There exists some $f\in S_d$ with $d>0$, such that the image of $x$ is contained in the affine open subset $D^{+}(f):=\Spec S_{(f)} \subset X$, the spectrum of the localization at $(f)$. Taking the fiber product from the morphisms $\rho$ and $x$, we get the diagram

\medskip

\begin{small}
\begin{tikzcd}
  \Spec(S[\frac{1}{f}]\otimes_{S_{(f)}} k )=\Spec(S[\frac{1}{f}])\times_{\Spec((S_{(f)})} \Spec(k)  \arrow[r] \arrow[d] \ar[dr, phantom, "\square"]
    & \Spec(k) \arrow[d, "x"] \\
  \Spec(S[\frac{1}{f}])=\rho^{-1}(D^+(f)) \arrow[r, "\rho"] \arrow[d, "\cap"]
& |[]|   D^+(f)=\Spec(S_{(f)}) \arrow[d, "\cap"]\\
 C \arrow[r, "\rho"]&  X
 \end{tikzcd}
 \end{small}
 \medskip

\noindent
with $C\times_{X} \Spec(k)= \Spec(S[\frac{1}{f}])\times_{\Spec((S_{(f)})} \Spec~(k)$.
We conclude using the following Lemma~\ref{Lemme-LMB} for the graded algebra $B=S[\frac{1}{f}]\otimes_{S_{(f)}} k$, whose degree zero homogeneous part is a field. Indeed, the $k$-rational point 
$x:\Spec(k) \longrightarrow D^{+}(f):=\Spec S_{(f)} \subset X$
corresponds to a morphism of rings $x^{\sharp} : S_{(f)} \rightarrow k$, whose kernel ${\mathcal M}$ is a maximal ideal of $S_{(f)}$. From the isomorphism induced by $x^{\sharp} : S_{(f)}/{\mathcal M} S_{(f)} \simeq k$, we deduce the isomorphism of graded rings
$$B=S[\frac{1}{f}]\otimes_{S_{(f)}} \left(S_{(f)}/{\mathcal M} S_{(f)}\right) \simeq S[\frac{1}{f}]/{\mathcal M}S[\frac{1}{f}],$$
whose degree zero homogeneous part is $S_{(f)}/{\mathcal M} S_{(f)}$, which is isomorphic to $k$ hence is a field.
\end{proof}

\begin{lemma}(Moret-Bailly)\label{Lemme-LMB}
Let $B=\bigoplus_{n\in {\mathbb Z}} B_n$ be a ${\mathbb Z}$-graded ring. Assume that $B_d$ contains an  element $f$ invertible in $B$, for some $d>0$.

\begin{enumerate}
\item Then, the morphism of ${\mathbb Z}$-graded rings
$$\begin{matrix}
\phi_f : & &B_0[t, t^{-1}]& \longrightarrow& B^{(d)}&=&\bigoplus_{n \in d{\mathbb Z}} B_n\\
~&~&t&\mapsto &f&~&~
\end{matrix}$$
is an isomorphism.
\item If moreover $B_0$ is a field and $d$ is minimal for the properties $d>0$ and $B_d\cap B^* \neq \emptyset$, then the composite map
$$B_0[t, t^{-1}] \overset{\phi_f}\longrightarrow B^{(d)} \hookrightarrow B \rightarrow B_{\hbox{red}}$$ 
is a graded ring  isomorphism.
\end{enumerate}

\end{lemma}

\begin{proof}
Let $m\in {\mathbb Z}$. Since $f\in B_d\cap B^*$, the restriction of $\phi_f$ to the homogeneous part of some degree $m\in {\mathbb Z}$
$$\begin{matrix}
& & B_0t^m&\longrightarrow& B_{dm}\\
~&~&b_0t^m&\mapsto&b_0f^m
\end{matrix}$$
is an isomorphism of ${\mathbb Z}$-modules, from which the first item follows.

\medskip

For the second item, we begin by proving that for any $e\notin d{\mathbb Z}$ and $g \in B_e$, we have $g^d=0_B$. Considering the Euclidean division $e=dq+r$ with $0<r<d$ of $e$ by $d$, we have that $gf^{-q} \in B_{e-dq}=B_r$ with $r>0$, so by minimality of $d$ we deduce that $gf^{-q}\notin B^*$. Since $f\in B^*$, it follows that $g\notin B^*$, and then that $g^df^e \notin B^*$. But $g^df^e \in B_{ed-de}=B_0$ which is a field, so $g^df^e=0_B$, hence $g^d=0_B$.  

Now, let $\frak N$ be the nilradical of $B$ and let $\pi : B \rightarrow B_{\hbox{red}}=B/{\frak N}$ be the canonical morphism. We have to prove, thanks to the first item, that the graded ring morphism
$$\pi_{\vert_{B^{(d)}}} : B^{(d)} \hookrightarrow B \overset{\pi}\rightarrow B/{\frak N}$$
is an isomorphism. 

The morphism $\pi$ is onto from $B= B^{(d)} \oplus \left(\bigoplus_{e\notin d{\mathbb Z}}B_e\right)$ to $B/{\frak N}$ and sends the right part $\bigoplus_{e\notin d{\mathbb Z}}B_e$ to $0_B$ by the previous paragraph, so $\pi$ remains onto from the first part $B^{(d)}$.

Now let  $h\in B^{(d)}\cap \Ker( \pi)$ and $b_{0,m}f^m$ be the homogeneous part of some degree $dm$ with $b_{0,m}\in B_0$. 
Since $\pi$ is a graded ring morphism,
 we have $b_{0,m}f^m\in\Ker(\pi)={\frak N}$. From $f \in B^*$ we deduce that $b_{0,m}\in {\frak N}\cap B_0$ is a  nilpotent element in the field $B_0$, hence is equal to zero. We conclude that $\pi_{\vert_{B^{(d)}}}$ is an isomorphism.
\end{proof}

Consider
a rational point of a weighted projective space over a finite field $k$ with $q$ elements. Starting from a rational representative
whose existence follows from Proposition \ref{LMB-projPond}, one can prove (see Lemma 7 in \cite{Marc}) that it has exactly  $q-1$ representatives in 
$k^{n+1}\setminus \{0\}$.
In particular  we have
$\sharp {\mathbb P}(a_0,\ldots,a_n)({\mathbb F}_q)=p_n.$

For many more details about weighted projective spaces, one can consult the article of Beltrametti and Robbiano (see \cite{Beltrametti-Robbiano}) for a theory over an algebraically closed field of characteristic 0, the article of Dolgachev (see \cite{Dolgachev}) for a theory over a field of characteristic prime to all the $a_i$'s, and the Appendix of  \cite{UCLA} for a survey of the different points of view.



\subsection{A lower bound}

Let $F$ be  a  homogeneous polynomial in $S$  of degree $d$, so that
$$F(\lambda^{a_0}X_0,\ldots,\lambda^{a_n}X_n)=\lambda^{d}F(X_0,\ldots,X_n)\ \ {\rm for\  all}\ \ \lambda \in {\overline{\mathbb F}}_q^{\ast}$$ 
and  let $V(F)$ be the  hypersurface defined by $F$ in ${\mathbb P}(a_0,\ldots,a_n)$.

We define, as in the introduction, the quantity:
$$e_q(d ; a_0,\ldots,a_n):=\max_{F\in S_d\setminus\{0\}}\sharp V(F)({\mathbb F}_q)$$
where $S_d$ stands for the space of weighted homogeneous polynomials in $S$  of weighted degree $d$.
Remark that the previous quantity is only defined for $d\in a_0{\mathbb N}+\cdots+a_n{\mathbb N}$.

Consider now the polynomial
$$F=\prod_{i=1}^{d/a_{rs}}(\alpha_iX_r^{a_{rs}/a_r} - \beta_iX_s^{a_{rs}/a_s})$$
where $r,s\in\{0,\ldots, n\}$ are distincts indices, $a_{rs}=\lcm(a_r,a_s)$, $d$ is a multiple of $a_{rs}$ satisfying $d\leq a_{rs}(q+1)$ and the $(\alpha_i,\beta_i)$'s are distinct elements of ${\mathbb P}^1({\mathbb F}_q)$.
It has been proved  in \cite{UCLA}  that $\sharp V(F)({\mathbb F}_q)=(d/a_{rs})q^{n-1}+p_{n-2}.$
So, if $a:=\min\{\lcm(a_r,a_s), 0\leq r<s\leq n\}$ and  $a\mid d$, then it implies that
$$e_q(d ; a_0,\ldots,a_n)\geq \min\{p_n,\frac{d}{a}q^{n-1}+p_{n-2}\}.$$




\section{Some morphisms between weighted projective spaces}\label{section-pullback}

\subsection{The morphisms $\pi_i$}
For $i=0,\ldots,n$, we consider  the following morphims $\pi_{i}$ : 

$$
\begin{matrix}

\pi_{i} &:& {\mathbb P}(a_0,\ldots,a_{i-1},1,a_{i+1},\ldots,a_n)&\longrightarrow & {\mathbb P}(a_0,\ldots,a_n)\cr
&&[x_0:\cdots:x_n]&\longmapsto & [x_0:\cdots:x_{i}^{a_{i}}:\cdots:x_n].\cr

\end{matrix}
$$

Our purpose in this Section is to study the behaviour of the rational points with respect to  these morphisms. For this purpose, let us fix some generator $\delta$ of the multiplicative group ${\mathbb F}_q^{\ast}$.

\medskip

For any given $i\in\{0,\ldots,n\}$, set $r_i=(a_i,q-1)$ the gcd of $a_i$ and $q-1$ and consider the map $\varphi_{a_i}$: 

$$
\begin{matrix}

\varphi_{a_i} &:& {\mathbb F}_q^{\ast}&\longrightarrow & {\mathbb F}_q^{\ast}\cr
&&z&\longmapsto & z^{a_i}.\cr

\end{matrix}
$$

Recall that the map $\varphi_{a_i}$ is a group homomorphism with kernel
$\Ker(\varphi_{a_i})=<\delta^{\frac{q-1}{r_i}}>=:\mu_{a_i}$,
the subgroup of ${\mathbb F}_q^{\ast}$ of $a_i$-th roots of unity in ${\mathbb F}_q^{\ast}$ which has order $r_i$,
and with image
$\Ima(\varphi_{a_i})=<\delta^{a_i}>=:\Delta^{a_i}$,
the subgroup of ${\mathbb F}_q^{\ast}$ of $a_i$-th powers which has order $\frac{q-1}{r_i}$.


Let $\mathcal P$ be the whole set of rational points over ${\mathbb F}_q$ of ${\mathbb P}(a_0,\ldots,a_n)$.
We have a partition ${\mathcal P}={\mathcal R}_i\cup {\mathcal T}_i\cup {\mathcal I}_i$ with respect to the $i$-th coordinate, where
 $${\mathcal R}_i:=\{[y_0:\cdots:y_n]\in {\mathbb P}(a_0,\ldots,a_n)({\mathbb F}_q) \mid y_i=0\}\cup \{{\mathcal O}_i\},$$

 $${\mathcal T}_i:=\{[y_0:\cdots:y_n]\in {\mathbb P}(a_0,\ldots,a_n)({\mathbb F}_q) \mid y_i=1\} \setminus \{{\mathcal O}_i\},$$
 
$${\mathcal I}_i:=\{[y_0:\cdots:y_n]\in {\mathbb P}(a_0,\ldots,a_n)({\mathbb F}_q) \mid y_i \in {\mathbb F}_q^{\ast}\setminus \Delta^{a_i}\}$$
and  ${\mathcal O}_i:=[0:\cdots:0:1:0:\cdots :0]$ is the point where $1$ appears at the index $i$.

Let us scrutinize more narrowly the sets ${\mathcal I}_i$ 
and ${\mathcal T}_i$.
In order to do this, consider,  for $j\in\{1, \ldots,q-1\}$, the sets
${\mathcal Z}_i(j)$ defined by
 $${\mathcal Z}_i(j):=\{[y_0:\cdots:y_n]\in {\mathbb P}(a_0,\ldots,a_n)({\mathbb F}_q) \mid y_i=\delta^j\}.$$

\begin{lemma}\label{LesZi}

We have:

\begin{enumerate}[label=(\roman*)]

\item 
${\mathcal Z}_i(j_1)={\mathcal Z}_i(j_2)$ if $j_1\equiv j_2 \pmod{r_i}$.

\item ${\mathcal Z}_i(r_i)={\mathcal T}_i$.

\item ${\mathcal I}_i=\emptyset$ if $r_i=1$ and 

$${\mathcal I}_i={\mathcal Z}_i(1)\cup \ldots \cup {\mathcal Z}_i(r_i-1)$$
otherwise.

\end{enumerate}
\end{lemma}

\begin{proof}
We begin by proving that $\delta^{r_i}=\lambda^{a_i}$ for some $\lambda \in {\mathbb F}_q^{\ast}$, which will be used in the proof of the three items. Indeed, there exist by B\'ezout Theorem some integers $u, v$ such that $r_i=ua_i+v(q-1)$, so that $\delta^{r_i}=(\delta^{u})^{a_i}\times (\delta^{q-1})^{v}=\lambda^{a_i}$ for $\lambda = \delta^{u}$. 

Suppose now that $j_2=j_1+mr_i$ for some integer $m$ and consider some $[y_0: \cdots: y_n] \in {\mathcal Z}_i(j_2)$. By writing $\delta^{r_i}=\lambda^{a_i}$, it is easily checked from $\delta^{j_2}=(\delta^{r_i})^m\times \delta^{j_1}=(\lambda^m)^{a_i}\times \delta^{j_1}$ that
$[y_0: \cdots: y_{i-1}: \delta^{j_2},y_{i+1}:\cdots,y_n]=[(\lambda^{-m})^{a_0}y_0: \cdots : (\lambda^{-m})^{a_{i-1}}y_{i-1} : \delta^{j_1}:(\lambda^{-m})^{a_{i+1}}y_{i+1}:\cdots:(\lambda^{-m})^{a_n}y_n]$
which lies in ${\mathcal Z}_i(j_1)$, so that ${\mathcal Z}_i(j_2) \subset {\mathcal Z}_i(j_1)$. The reverse inclusion follows similarly.

The second item can be proved likewise by writing $\delta^{r_i}=\lambda^{a_i}$, since then
$[y_0:\cdots: y_{i-1}: \delta^{r_i}:y_{i+1}:\cdots:y_n]=[(\lambda^{-1})^{a_0}y_0: \cdots: (\lambda^{-1})^{a_{i-1}}y_{i-1}: 1:(\lambda^{-1})^{a_{i+1}}y_{i+1}:\cdots:(\lambda^{-1})^{a_n}y_n]$.

Finally, the set ${\mathcal I}_i$ contains of course the union ${\mathcal Z}_i(1)\cup \ldots \cup {\mathcal Z}_i(r_i-1)$. Conversely, given some
$P=[y_0:\cdots:y_{i-1}:\delta^{h}:y_{i+1}:\cdots:y_n]\in{\mathcal I_i}$ with $1\leq h\leq q-1$ and $h$ not divisible by $a_i$,  then writing the Euclidean division of $h$ by $r_i$ gives the existence of integers $m$ and $j$ such that $h=r_im+j$ with $0\leq j\leq r_i-1$. Thus, writing 
$\delta^h=(\delta^{r_i})^m\times \delta^j=(\lambda^m)^{a_i}\times \delta^j$,  we get
$[y_0:\cdots : y_{i-1}:\delta^{h}:y_{i+1}:\cdots:y_n]=
[(\lambda^{-m})^{a_0}y_0 : \cdots : (\lambda^{-m})^{a_{i-1}}y_{i-1} : \delta^j :(\lambda^{-m})^{a_{i+1}}y_{i+1}:\cdots:(\lambda^{-m})^{a_{n}}y_n]$, so that $P\in {\mathcal Z}_i(j)$ for this $j\in \{1, \cdots, r_i-1\}$ which concludes the proof.

\end{proof}

The following proposition describes the number of pre-images of points by the morphism $\pi_i$ according to the set of the partition that they belong to.

\begin{proposition}\label{antecedent}
Let $P$ be a rational point of ${\mathbb P}(a_0,\ldots,a_n)$.
\begin{enumerate}[label=(\roman*)]

\item If $P\in {\mathcal R}_i$ then $P$ has exactly one  pre-image rational over ${\mathbb F}_q$ by $\pi_i$.

\item If $P\in {\mathcal T}_i$ then $P$ has exactly $r_i$  pre-images rational over ${\mathbb F}_q$ by $\pi_i$.

\item If $P\in {\mathcal I}_i$ then $P$ has no  pre-image rational over ${\mathbb F}_q$ by $\pi_i$.
\end{enumerate}
\end{proposition}

\begin{proof}
$(i)$ The point ${\mathcal O}_i:=[0:\cdots:0:1:0:\cdots :0]\in {\mathbb P}(a_0,\ldots,a_n)$ has only one pre-image by $\pi_i$, namely the point $[0:\cdots:0:1:0:\cdots :0]\in {\mathbb P}(a_0,\ldots,a_{i-1},1,a_{i+1},\ldots,a_n)$.
Moreover, the point $[y_0:\cdots:y_{i-1}: 0 :y_{i+1}:\cdots:y_n]$ has only one pre-image by $\pi_i$, that is  the point 
$[y_0:\cdots:y_{i-1}: 0 :y_{i+1}:\cdots:y_n]$.

$(ii)$ The point $[y_0:\cdots:y_{i-1}: 1 :y_{i+1}:\cdots:y_n]$ has $r_i$ pre-images by $\pi_i$, which are precisely the points 
$[y_0:\cdots:y_{i-1}: \delta^{\frac{(q-1)k}{r_i}} :y_{i+1}:\cdots:y_n]$
 for $k=1,\dots,r_i$ (the elements $\delta^{\frac{(q-1)k}{r_i}}$ are the $a_i$-th roots of unity in ${\mathbb F}_q^{\ast}$ i.e. the elements of the group $\mu_{a_i}$).

$(iii)$ The points  $[y_0:\cdots:y_n]$ with
$y_i\not\in\Delta^{a_i}$
have no rational pre-image by $\pi_i$ since $y_i$ is not a $a_i$-th power in ${\mathbb F}_q^{\ast}$.
\end{proof}



\subsection{Number of zeros of the pullback}

Let $F$ be a homogeneous polynomial in ${\mathbb F}_q[X_0,\ldots,X_n]$ of  $(a_0,\ldots,a_n)$-weighted degree $d\leq q+1$, i.e.
$$F(\lambda^{a_0} X_0, \ldots,\lambda^{a_n} X_n)=\lambda^d F(X_0,\ldots,X_n)$$
for any $\lambda\in {\overline{\mathbb F}}_q^{\ast}$.
Let 
$$\pi_{i}^{\ast}F(X_0,\ldots,X_n):=(F\circ\pi_{i})(X_0,\ldots,X_n)=F(X_0,\ldots, X_{i}^{a_i}, \ldots,X_n)$$
 be the pullback of $F$, an homogeneous polynomial of 
$(a_0,\ldots,a_{i-1},1,a_{i+1},\ldots,a_n)$-weighted degree $d$.
We consider the hypersurface $V_{{\mathbb P}(a_0,\ldots,a_n)}(F)$ of zeros of $F$ in ${\mathbb P}(a_0,\ldots,a_n)$ 
whose number of rational points over ${\mathbb F}_q$ is denoted by $N(F)$.
We also consider the hypersurface $V_{{\mathbb P}(a_0,\ldots,a_{i-1},1,a_{i+1},\ldots,a_n)}(\pi^{\ast}F)$ of zeros of $\pi^{\ast}F$ in ${\mathbb P}(a_0,\ldots,a_{i-1},1,a_{i+1},\ldots,a_n)$ 
whose number of rational points over ${\mathbb F}_q$ is denoted by $N(\pi_{i}^{\ast}F)$.

Let us set:
$$A(F):=\sharp(V_{{\mathbb P}(a_0,\ldots,a_n)}(F)\cap {\mathcal A})$$
for ${\mathcal A}\in\{{\mathcal R}_i,{\mathcal T}_i,{\mathcal I}_i, {\mathcal Z}_i(j)\}$.
So, $N(F)$ denotes the number of rational points of $V_{{\mathbb P}(a_0,\ldots,a_n)}(F)$ and $R_i(F), T_i(F), I_i(F)$ and $Z_i(j)(F)$ denote the number of those rational points  lying on ${\mathcal R}_i, {\mathcal T}_i, {\mathcal I}_i$ and ${\mathcal Z}_i(j)$ respectively.

\begin{proposition}\label{identites}
We have :

\begin{enumerate}[label=(\roman*)]
\item $$N(F)=R_i(F)+T_i(F)+I_i(F).$$

\item
$$N(\pi_i^{\ast}F)=r_iT_i(F)+R_i(F).$$

\item Consider the automorphism $\sigma_i : [y_0:\cdots:y_n]
 \longmapsto [y_0:\cdots:y_{i-1}: \delta y_i :y_{i+1}:\cdots:y_n]$ of ${\mathbb P}(a_0,\ldots,a_n)$.
If $r_i:=(a_i,q-1)\not=1$ then:
	\begin{enumerate}
	\item for $j=1,\ldots,r_i-1$, we have $T_i(F\circ\sigma_i^j)=Z_i(j)(F)$,
	\item for $j=r_i-1$, we have $T_i(F\circ\sigma_i^j)=T_i(F)$
	\item and $R_i(F)=R_i(F\circ\sigma_i^j)$ for $1\leq j\leq r_i-1$.
	\end{enumerate}
\end{enumerate}
\end{proposition}

\begin{proof}
The first equality comes from the partition ${\mathcal P}={\mathcal R}_i\cup {\mathcal T}_i\cup {\mathcal I}_i$.

The second one  from Proposition \ref{antecedent} 
and the fact that if $P$ is a rational point over ${\mathbb F}_q$ of $V_{{\mathbb P}(a_0,\ldots, a_{i-1},1,a_{i+1},\ldots,a_n)}(\pi^{\ast}F)$ then $\pi_i(P)$ is a point of $V_{{\mathbb P}(a_0,\ldots,a_n)}(F)$ which is rational over 
${\mathbb F}_q$.

The third one follows from the fact that the automorphism $\sigma_i$ sends ${\mathcal T}_i$ to ${\mathcal Z}_i(1)$ and ${\mathcal Z}_i(j)$ to  ${\mathcal Z}_i(j+1)$ for $1\leq j \leq r_i-1$, and by Lemma \ref{LesZi}  sends  ${\mathcal Z}_i(r_i-1)$ to ${\mathcal T}_i$, and leaves ${\mathcal R}_i$ stable.
\end{proof}

Now we are enable  to prove a
 relation on the numbers of points between two floors.

\begin{proposition}\label{Mondo}
 Let $F$ be a homogeneous polynomial in ${\mathbb F}_q[X_0,\ldots,X_n]$ with respect to the weights $(a_0,a_1,\ldots,a_n)$. 
 For  $i\in\{0,\ldots,n\}$, let 
 $$
\begin{matrix}
\pi_{i} &:& {\mathbb P}(a_0,\ldots,a_{i-1},1,a_{i+1},\ldots,a_n)&\longrightarrow & {\mathbb P}(a_0,\ldots,a_n)\cr
&&[x_0:\cdots:x_n]&\longmapsto & [x_0:\cdots:x_{i}^{a_{i}}:\cdots:x_n]\cr
\end{matrix}
$$
and $\pi_{i}^{\ast}F(X_0,\ldots,X_n):=(F\circ\pi_{i})(X_0,\ldots,X_n)=F(X_0,\ldots, X_{i}^{a_i},\ldots,,X_n)$ be the pullback of $F$. 

Let also $\delta$ be a primitive element of ${\mathbb F}_q^{\ast}$, and
$\sigma _i : [y_0:\cdots:y_n]
 \longmapsto [y_0:\cdots:y_{i-1}:\delta y_i : y_{i+1}:\cdots:y_n]$ inside ${\mathbb P}(a_0,\ldots,a_n)$.
 Denote by $r_i=(a_i,q-1)$  the gcd of $a_i$ with $q-1$.

 Then, the number $N(F)$ of rational points over ${\mathbb F}_q$ of the hypersurface of the weighted projective space 
${\mathbb P}(a_0,a_1,\ldots,a_n)$ defined by $F$ satisfies
$$N(F)\leq \frac{1}{r_i}\sum_{j=0}^{r_i-1}N(\pi_{i}^{\ast}(F\circ\sigma_i^j)).$$ 

\end{proposition}

\begin{proof}
If $r_i=1$, then the set $I_i$ is empty and by
$(i)$ and $(ii)$ of Proposition \ref{identites}, we have
$N(F)=R_i(F)+T_i(F)=N(\pi_i^{\ast}F)$ which gives the result.

Suppose now that $r_i\not=1$. By $(i)$  of Proposition \ref{identites}, we have:
$$r_iN(F)=(r_iT_i(F)+R_i(F)) + (r_iI_i(F)+(r_i-1)R_i(F)).$$

On one hand, we have by $(ii)$ of Proposition \ref{identites} that
$r_iT_i(F)+R_i(F)=N(\pi_i^{\ast}F)$ 
and 
on the other hand, by Lemma \ref{LesZi}, we can write
$I_i(F)\leq\sum_{j=1}^{r_i-1}Z_i(j)(F)$. Thus, we have:
\begin{align*}
r_iI_i(F)+(r_i-1)R_i(F) & \leq r_i\left(\sum_{j=1}^{r_i-1}Z_i(j)(F)\right) +(r_i-1)R_i(F)\\
&=\sum_{j=1}^{r_i-1}\left(r_iZ_i(j)(F)+R_i(F)\right).
\end{align*}
Moreover, by  Proposition \ref{identites} $(iii)$, we have:
$$r_iZ_i(j)(F)+R_i(F)=r_iT_i(F\circ\sigma_i^j)+R_i(F\circ\sigma_i^j)$$
and we obtain with Proposition \ref{identites} $(ii)$:
$$r_iZ_i(j)(F)+R_i(F)=N(\pi_i^{\ast}(F\circ\sigma_i^j)).$$
Thus we deduce that:
$$r_iI_i(F)+(r_i-1)R_i(F)=\sum_{j=1}^{r_i-1}N(\pi_i^{\ast}(F\circ\sigma_i^j))$$
and we obtain the desired formula.

\end{proof}

 \begin{remark}\label{remark-egaux}
 Note that under the additional assumption that $(a_i, a_j)=1$ for any $1\leq i\neq j\leq n$, we have equality in the above Proposition~\ref{Mondo}. This comes from the fact that, under this assumption, the sets ${\mathcal Z}_i(j)$ for $1\leq j\leq r_i-1$  form a partition of ${\mathcal I}_i$, hence both inequalities in the above proof are equalities. 
 It remains to show that the sets ${\mathcal Z}_i(j)$ for $1\leq j\leq r_i-1$ are pairwise disjoint.
 Indeed, suppose that there is some common point with ${\mathbb F}_q$-coordinates
 $$[y_0: \cdots : y_{i-1} : \delta^{j_1} : y_{i+1} : \cdots : y_n]=[y'_0 : \cdots : y'_{i-1} : \delta^{j_2} : y'_{i+1} : \cdots : y'_n]$$
 inside ${\mathcal Z}_i(j_1)\cap {\mathcal Z}_i(j_2),$
 with say $1 \leq j_1 \leq j_2\leq r_i-1$. Since this point does not lie in ${\mathcal R}_i$, there is at least one position $k\neq i$, such that $y_k\neq 0\neq y'_k$. Since they are equal, there is some $\lambda \in {\overline{\mathbb F}}_q^{\ast}$ such that
$$(y'_0, \cdots, y'_{i-1}, \delta^{j_2}, y'_{i+1}, \cdots, y'_n)=(\lambda^{a_0}y_0, \cdots, \lambda^{a_{i-1}}y_{i-1}, \lambda^{a_i}\delta^{j_1}, \lambda^{a_{i+1}}y_{i+1}, \cdots, \lambda^{a_n}y_n).$$

  Looking at  the $k$-th and the $i$-th position, we get $y'_k=\lambda^{a_k}y_k$ and $\delta^{j_2}=\lambda^{a_i}\delta^{j_1}$. It follows first that 
  $\lambda^{a_k} = \frac{y'_k}{y_k} \in {\mathbb F}_q^{\ast}$, second that $\lambda^{a_i}=\delta^{j_2-j_1} $. But from a B\'ezout relation $ua_k+va_i=1$, we deduce that 
  $$\lambda = (\lambda^{a_k})^u\times (\lambda^{a_i})^v= (\frac{y'_k}{y_k})^u\times (\delta^{j_2-j_1} )^v \in {\mathbb F}_q^{\ast}.$$

Hence, we have $\lambda =\delta^m$ for some $m\in {\mathbb N}$, so that $\delta^{j_2-j_1}=\lambda^{a_i}=\delta^{ma_i}$. It follows that
$j_2-j_1 \equiv ma_i \hbox{~(mod $q-1$)}$. Since $r_i=(a_i, q-1)$ divides both $a_i$ and $q-1$, it divides $j_2-j_1 \in \{0, \cdots, r_i-1\}$, hence $j_1=j_2$ and we are done.
\end{remark}





\section{An upper bound for the number of rational points}\label{section-upper-bound}

We prove in this section that an hypersurface in a weighted projective space cannot have more rational points than in a standard projective space. The proof is based on an unscrewing and uses Proposition \ref{Mondo}.

\begin{figure}[h]\label{diagram}
\begin{tikzpicture}
[node distance=1.5cm]

 \node (Bas)          {${\mathbb P}(a_0,a_1,a_2,\ldots,a_n)$};
 \node (Milieu)   [above of=Bas]   {${\mathbb P}(1,a_1,a_2\ldots,a_n)$};
 \node (vide2) [above of=Milieu] {$\vdots$};
\node (Haut)   [above of=vide2]   {${\mathbb P}(1,1,1,\ldots,1)={\mathbb P}^n$};

\draw[<-] (Bas) to node[left, midway,scale=0.9]  {$\pi_0$} (Milieu);
\draw[<-] (Milieu) to node[left, midway,scale=0.9]  {$\pi_1$} (vide2);
\draw[<-] (vide2) to node[left, midway,scale=0.9]  {$\pi_n$} (Haut);

 \end{tikzpicture}
\caption{Screwing of weighted projective spaces}
 \end{figure}

\begin{theorem}\label{ilenmer}
Let $F$ be a homogeneous polynomial in ${\mathbb F}_q[X_0,\ldots,X_n]$ of  $(a_0,a_1,\ldots,a_n)$-weighted degree $d\leq q+1$. Then the number $N(F)$ of rational points over ${\mathbb F}_q$ of the hypersurface of the weighted projective space ${\mathbb P}(a_0,a_1,\ldots,a_n)$ given by  the set of zeros of $F$ satisfies:
$$N(F)\leq dq^{n-1}+p_{n-2}.$$
\end{theorem}

\begin{proof}
Let $F$ be a homogeneous polynomial in ${\mathbb F}_q[X_0,\ldots,X_n]$ of  $(a_0,a_1,\ldots,a_n)$-weighted degree $d$. We consider the successive pullbacks
$\pi_0^{\ast}(F\circ\sigma_0^{j_0})$ 
with $j_0\in\{0,\ldots, r_0-1\}$, 
and $\pi_1^{\ast}(\pi_0^{\ast}(F\circ\sigma_0^{j_0})\circ\sigma_1^{j_1})$
with $j_1\in\{0,\ldots, r_1-1\}$, and so on, of $F$.

 By Proposition \ref{Mondo}, considering the morphism 
 $$
\begin{matrix}
\pi_{0} &:& {\mathbb P}(1,a_1,\ldots,a_n)&\longrightarrow & {\mathbb P}(a_0, a_1, \ldots,a_n)\cr
&&[x_0:x_1:\cdots:x_n]&\longmapsto & [x_0^{a_0}:x_1:\cdots:x_n]\cr
\end{matrix}
$$
 we have:
 $$
N(F)\leq \frac{1}{r_0}\sum_{j_0=0}^{r_0-1}N(F_0(j_0))$$
where
$F_0(j_0)=\pi_0^{\ast}(F\circ\sigma_0^{j_0})$.
Then, considering the morphism 
 $$
\begin{matrix}
\pi_{1} &:& {\mathbb P}(1, 1, a_2\ldots,a_n)&\longrightarrow & {\mathbb P}(1,a_1,\ldots,a_n)\cr
&&[x_0:x_1:x_2:\cdots:x_n]&\longmapsto & [x_0:x_1^{a_1}:x_2:\cdots:x_n]\cr
\end{matrix}
$$
 we have for $0\leq j_0\leq r_0-1$:
$$
N(F_0(j_0))\leq\frac{1}{r_1}\sum_{j_1=0}^{r_1-1}N(F_1(j_1))$$
where
$F_1(j_1)=\pi_1^{\ast}(F_0(j_0)\circ\sigma_1^{j_1})$.
 
 Thus:
 $$
N(F)\leq\frac{1}{r_0r_1}\sum_{j_0=0}^{r_0-1}
\sum_{j_1=0}^{r_1-1}N(F_1(j_1)).$$

Continuing this process, we obtain
 $$
N(F)\leq\frac{1}{r_0\ldots r_n}\sum_{j_0=0}^{r_0-1} \ldots
\sum_{j_n=0}^{r_n-1}N(F_n(j_n)).$$

 The last polynomials are  homogeneous polynomials of degree $d$ in the standard $n$-dimensional projective space ${\mathbb P}^n={\mathbb P}(1,\ldots,1)$.
 Then we apply the Serre bound 
 $$N(F)\leq \frac{1}{r_0\ldots r_n}r_0\ldots r_n (dq^{n-1}+p_{n-2})=dq^{n-1}+p_{n-2}$$
 and we get the result.
\end{proof}




\section{The main result}\label{section-main-result}

We are now enable to state and prove Conjecture~\ref{conjecture} provided $a_1=1$ (it was already assumed in the conjecture that $a_0=1$).

\begin{theorem}
For any degree $d$ and for any nonnegative integers $a_2,\ldots,a_n$, we have:
$$e_q(d ; 1,1,a_2,\ldots,a_n)=\min\{p_n,dq^{n-1}+p_{n-2}\}.$$
In other words, Conjecture~\ref{conjecture}  is true for any $(a_1,a_2,\ldots,a_n)$ with $a_1=1$ and without any assumption on the degree $d$.
\end{theorem}

\begin{proof}
As seen in Subsection \ref{section-wps}, a hypersurface of ${\mathbb P}(1,1,a_2\ldots,a_n)$ 
has obviously 
a number of rational points less than or equal to $p_n$
and the hypersurface defined by 
the homogeneous polynomial $X_0^{d-q-1}(X_0^qX_1-X_0X_1^q)$  
of degree $d\geq q+1$
has $p_n$ rational points  (the degree is equal to $d$ since we have supposed that the weights of $X_0$ and $X_1$ are equal to 1 in the graded ring ${\mathbb F}_q[X_0,\ldots,X_n]$).
Now if $d\leq q+1$, by Theorem \ref{ilenmer}  we have $e_q(d ; 1,1,a_2,\ldots,a_n)\leq \min\{p_n,dq^{n-1}+p_{n-2}\}$
and  the bound  is met using the following degree $d$ homogeneous polynomial:
$$F=\prod_{i=1}^d(\alpha_i X_0 - \beta_i X_1)$$
where $(\alpha_1:\beta_1), \ldots, (\alpha_d:\beta_d)$ are distinct elements of ${\mathbb P}^1({\mathbb F}_q)$.
\end{proof}

\bigskip

{\bf Acknowledgments:}
The authors are very grateful to Fabien Herbaut for fruitful discussions.
They would  like also to thank Jade Nardi, Sudhir Ghorpade and Mrinmoy Datta
for some comments on this question. 
Finally, they would like to express their gratitude to Laurent Moret-Bailly for providing them a proof
of 
Proposition \ref{LMB-projPond}
and the more general statement given in Proposition \ref{Prop-LMB_Scheme}.




\bigskip
\bibliographystyle{plain}

\begin{thebibliography}{}

\end{thebibliography}


\begin{thebibliography}{1}

\bibitem{UCLA}
Y. Aubry, W. Castryck, S. Ghorpade, G. Lachaud, M. O'Sullivan and S. Ram,
\newblock Hypersurfaces in weighted projective spaces over finite fields with applications to coding theory,
\newblock {\em Algebraic geometry for coding theory and cryptography}, 25--61, Assoc. Women Math. Ser., 9, Springer, Cham, 2017.

\bibitem{B-D-G-2018}
P. Beelen, M. Datta and S. Ghorpade,
\newblock Maximum number of common zeros of homogeneous polynomials over finite fields,
\newblock {\em Proc. Amer. Math. Soc.} 146 (2018), no. 4, 1451--1468.  

\bibitem{B-D-G-2022}
P. Beelen, M. Datta and S. Ghorpade,
\newblock A combinatorial approach to the number of solutions  of homogeneous polynomial equations over finite fields,
\newblock {\em Moscow Math. Journal} Vol. 22, Number 4, October-December 2022, Pages 565--593.

\bibitem{Beltrametti-Robbiano}
M. Beltrametti and L. Robbiano,
\newblock Introduction to the theory of weighted projective spaces,
\newblock {\em Expo. Math.} 4 (1986), 111--162.

\bibitem{Bogu}
M. Boguslavsky,
\newblock On the number of solutions of polynomial systems,
\newblock {\em Finite Fields Appl.} 3 (1997), no. 4, 287--299.


\bibitem{D-G-2015}
M. Datta and S. Ghorpade,
\newblock On a conjecture of Tsfasman and an inequality of Serre for the number of points of hypersurfaces over finite fields,
\newblock {\em Moscow Math. Journal} Vol. 15, Number 4, October-December 2015, Pages 715--725. 

\bibitem{D-G-2017}
M. Datta and S. Ghorpade,
\newblock Number of solutions of systems of homogeneous polynomial equations over finite fields,
\newblock {\em Proc. Amer. Math. Soc.} 145 (2017), no. 2, 525--541.  


\bibitem{Delorme}
C. Delorme,
\newblock Espaces projectifs anisotropes,
\newblock {\em Bull. Soc. Math. France.} 103 (1975), no. 2, 203--223.  



\bibitem{Dolgachev}
I. Dolgachev,
\newblock Weighted projective varieties,
\newblock {\em Group Actions and Vector Fields (Vancouver, B.C., 1981),} (J. B. Carell, ed.), Lecture Notes in Mathematics, vol. 956, Springer, Berlin, 1982, pp. 34--71.



\bibitem{Marc}
M. Perret,
\newblock On the number of points of some varieties over finite fields,
\newblock {\em Bull. London Math. Soc.}, 35 (2003), no. 3, 309--320.

\bibitem{Serre}
J. -P. Serre,
\newblock Lettre \`a M. Tsfasman,
\newblock {\em Journ\'ees Arithm\'etiques}, 1989, (Luminy, 1989), Ast\'erisque, vol. 198-200, Soci\'et\'e Math\'ematique de France, Paris, 1991, pp. 351--353.


\bibitem{Sorensen}
A. B. S\o rensen,
\newblock Projective Reed-Muler codes,
\newblock {\em IEEE Trans. Inform. Theory} 37 (1991), no. 6, 1567--1576.

\end{thebibliography}

\end{document}